\documentclass[10pt]{amsart}

\usepackage{amssymb,amsmath,geometry,latexsym,graphics,tabularx,shapepar,enumerate,
hyperref,stmaryrd}

\newcommand{\ZZ}{\mathbb{Z}}

\newcommand{\CC}{\mathbb{C}}
\newcommand{\QQ}{\mathbb{Q}}

\newcommand{\NN}{\mathbb{N}}

\newcommand{\sqm}[4]{\displaystyle{
\left({#1 \atop #3}{#2 \atop #4}\right)}}

\newtheorem{Theorem}{Theorem}[section]
\newtheorem{Lemma}[Theorem]{Lemma}
\newtheorem{Proposition}[Theorem]{Proposition}

\newtheorem{Definition}[Theorem]{Definition}
\newtheorem{Remark}[Theorem]{Remark}
\newtheorem{Conjecture}[Theorem]{Conjecture}

\title[Extremal quasi-modular forms]{On extremal quasi-modular forms after Kaneko and Koike}
\author{F. Pellarin, with an appendix by G. Nebe}
\address{Gabriele Nebe\\
	Lehrstuhl D f\"ur Mathematik \\
	RWTH Aachen University \\
52056 Aachen, Germany\\
Federico Pellarin\\
Institut Camille Jordan, UMR 5208\\
Site de Saint-Etienne \\
23 rue du Dr. P. Michelon \\
42023 Saint-Etienne, France}
\date{12/09/2019}
\begin{document}

\maketitle


\begin{small}
\noindent\textbf{Abstract.} Kaneko and Koike introduced the notion of extremal quasi-modular form and 
proposed conjectures on their arithmetic properties. The aim of this note is to prove a rather sharp multiplicity estimate for these quasi-modular forms. The note ends with discussions and partial answers around these conjectures and an appendix by G. Nebe containing the proof of the integrality of the Fourier coefficients of the normalised extremal quasimodular form of weight $14$ and depth $1$.
\end{small}

\medskip

\setcounter{tocdepth}{1}

\section{Introduction}
In \cite{KK2}, Kaneko and Koike introduced the notion of {\em extremal quasi-modular form} and discussed
the potential arithmetic interest of these functions. Let $E_2(z),E_4(z),E_6(z)$ be the Eisenstein series of weights $2,4,6$ respectively, normalised to have 
limit one as the imaginary part $\Im(z)$ of $z$, the variable in the complex upper-half plane $\mathcal{H}$, tends to infinity.  
These functions have quite explicit series expansions in $\ZZ[[q]]$ in terms of the uniformiser $q=e^{2\pi iz}$,
convergent for $|q|<1$:
\begin{equation}\label{E2E4E6}
E_2(z)=1-24\sum_{n\geq 1}\sigma_1(n)q^n,\quad E_4(z)=1+240\sum_{n\geq 1}\sigma_3(n)q^n,\quad
E_6(z)=1-504\sum_{n\geq 1}\sigma_5(n)q^n,
\end{equation} where $\sigma_k(n)$ denotes the sum $\sum_{d|n}d^k$,
over the positive divisors of $n$. It is well known and easy to prove that these functions are
algebraically independent over $\CC$ and the three dimensional polynomial algebra $\widetilde{M}:=\CC[E_2,E_4,E_6]$
is graded by the respective weights $2,4,6$ of $E_2,E_4,E_6$.
A {\em quasi-modular form of weight $w$} is a polynomial in $E_2,E_4,E_6$ homogeneous of weight $w$. In particular such a function $f$, when non-zero, can be 
written in a unique way as
$$f=f_0+f_1E_2+\cdots+f_lE_2^l$$ with $f_i$ a modular form of weight $w-2i$ for all $i$ and $f_l\neq0$.
We refer to the integer $l$ as to the {\em depth} of $f$. Let $\widetilde{M}^{\leq l}_{w}$ be the finite-dimensional vector space
of quasi-modular forms of weight $w\in2\ZZ$ and depth $\leq l$.
We have 
$$\widetilde{M}=\CC[E_2,E_4,E_6]=\bigcup_{l\geq 0}\bigoplus_w\widetilde{M}^{\leq l}_{w}.$$
By (\ref{E2E4E6}) the $\CC$-algebra $\widetilde{M}$ embeds in the 
subring of series of $\CC[[q]]$ which are converging for $|q|<1$. From now on, we identify 
quasi-modular forms with the formal series representing their $q$-expansions (Fourier series expansions). If $f\in\CC[[q]]$,
we set
$$\nu(f):=\text{ord}_{q=0}(f)\in\NN\cup\{\infty\},$$ where $\NN=\{n\in\ZZ:n\geq 0\}$. This is the order of vanishing at $q=0$ of the Fourier series of $f$ and defines a valuation over the $\CC$-algebra $\widetilde{M}$.

The dimension $\delta_l(w)$ of $\widetilde{M}^{\leq l}_{w}$ can be computed in the following way:
$$\delta_l(w)=\sum_{i=0}^ld(w-2i),$$
where $d(w)$ denotes, for $w\in2\ZZ$, the dimension of the $\CC$-vector space $M_w$ of {\em modular forms}
of weight $w$ (which is also equal to $\widetilde{M}^{\leq 0}_{w}$). We recall that 
\begin{equation}\label{d-w}
d(w)=\left\{\begin{matrix}0 \text{ if } w<0\\ \left\lfloor\frac{w}{12}\right\rfloor\text{ if }w\geq 0\text{ and }w\equiv2\pmod{12}\\ \left\lfloor\frac{w}{12}\right\rfloor+1\text{ in all the other cases.}\end{matrix}\right.
\end{equation}
If $w-2l=2$ with $w$ even, $\widetilde{M}^{\leq l}_w$ is non-zero but there are no quasi-modular forms of weight $w$ and depth $l$ because $d(2)=0$.

The following definition is due to Kaneko and Koike, see \cite{KK2}.
\begin{Definition}\label{def-1}{\em A quasi-modular form $f$ of weight $w$ and depth $l$ is {\em extremal} if 
$$\nu(f)=\delta_l(w)-1.$$}\end{Definition}
In \cite{KK2} Kaneko and Koike address the following question.

\medskip

{\em Does there always exist an extremal quasimodular form of given weight $w$ and depth $l$, provided $w$ and $l$ satisfy the necessary constraint $0 \leq 2l \leq w$, $2l\neq w-2$? And is it unique when normalized?} 

\medskip

We cannot answer this question in full generality, but in the present note we show that the answer is affirmative if we suppose that $l\leq 4$.

Assume that $\widetilde{M}^{\leq l}_w\neq\{0\}$. We set:
$$\nu_{\max}(l,w):=\max\{n\in\NN\text{ such that there exists }f\in \widetilde{M}^{\leq l}_w\setminus\{0\}\text{ with }\nu(f)=n\}\in\NN.$$ Note that 
\begin{equation}\label{l-b}\nu_{\max}(l,w)\geq \delta_l(w)-1.\end{equation}
To see this we set $\delta=\delta_l(w)$ and we consider a basis $(g_0,\ldots,g_{\delta-1})$ of $\widetilde{M}^{\leq l}_w$. Writing the $q$-expansion of each element of the basis $g_i=\sum_{j\geq 0}g_{i,j}q^j\in\CC[[q]]$ the matrix
\begin{equation}\label{matrix-gij}
(g_{i,j})_{\begin{smallmatrix}0\leq i\leq \delta-1\\ 0\leq j\leq \delta-2\end{smallmatrix}}
\end{equation} has rank $\leq \delta-1$ which justifies (\ref{l-b}).

\begin{Definition}\label{def-2}{\em A quasi-modular form $f$ of weight $w$ and depth $\leq l$ is {\em analytically extremal} if 
$$\nu(f)=\nu_{\max}(l,w).$$}\end{Definition}
For any $l,w$ such that $\widetilde{M}^{\leq l}_w\neq\{0\}$ there exists a unique normalised analytically extremal quasi-modular form $f_{l,w}\in \widetilde{M}^{\leq l}_{w}\setminus\{0\}$. We may reinforce the terminology by saying that an extremal form in the sense of Definition \ref{def-1} is {\em algebraically extremal}. The attribute 'algebraic' is chosen because the definition seems to rather involve the 
 algebraic structure of the $\CC$-vector space $\widetilde{M}^{\leq l}_{w}$.

\begin{Theorem}\label{firsttheorem}
Let $l$ be in the set $\{0,1,2,3,4\}$ and assume that $0 \leq 2l \leq w$, $2l\neq w-2$.
Then a non-zero quasi-modular form of weight $w$ and depth $l$ is algebraically extremal if and only if it is analytically extremal.
\end{Theorem}

This is a direct consequence of Theorem \ref{thirdtheorem} below and answers Kaneko and Koike's question from above, in the case $l\leq 4$. 
For $l$ in this range, the definitions \ref{def-1} and \ref{def-2} therefore coincide and there is no need to specify 'algebraically' or 'analytically' as an attribute for extremality. 

In \cite{KK2}, Kaneko and Koike ask whether an element in $\widetilde{M}_k^{\leq l}$ is uniquely determined by its first $\delta_l(k)$ Fourier coefficients and if one can prescribe these coefficients arbitrarily. In other words, 
they ask if there exists a basis of $\widetilde{M}_k^{\leq l}$ the $q$-expansions of the elements of which agree with the canonical diagonals basis $(1,q,q^2,\ldots)$ of the vector space $\CC[[q]]$ on its first $\delta_l(w)$ elements
(we may call this a 'diagonal basis' of $\widetilde{M}_k^{\leq l}$). For $l=0$ we have the so-called Miller bases which answer positively to this question; they can be easily constructed by the fact that the algebra of modular forms has dimension $2$. For higher depth we do not know a general answer to the question but Theorem \ref{thirdtheorem} implies that such bases exist for $l\leq 4$. 
Indeed, completing the matrix in (\ref{matrix-gij}) by adding a column, $U=(f_{i,j})_{0\leq i,j\leq \delta-1}$ is non-singular, because otherwise we would be able to construct, by linear combination, a non-zero element $f\in\widetilde{M}_k^{\leq l}\setminus\{0\}$ such that 
$\nu(f)>\nu_{\max}(l,w)=\delta_l(w)-1$, which is impossible (the last equality follows from Theorem \ref{firsttheorem}). Hence 'diagonal bases' for $\widetilde{M}_k^{\leq l}$, $l\leq 4$, exist. 

\section{A multiplicity estimate}

How large can be $\nu_{\max}(l,m)$? Classically, a simple {\em multiplicity estimate} holds on $f\in\widetilde{M}^{\leq l}_w$. For example it is easy to show, by using an elementary resultant argument that:
$$\nu_{\max}(l,w)\leq3\delta_l(w).$$ The main result of this section is Theorem \ref{thirdtheorem}, where we prove a rather sharp multiplicity estimate for quasi-modular forms.
For $w\in \ZZ$ we consider the difference $$\kappa_l(w):=d((l+1)w)-\delta_l(w)=\dim_\CC(M_{(l+1)w})-\dim_\CC(\widetilde{M}^{\leq l}_w)\in\ZZ.$$ We have the following elementary lemma.
\begin{Lemma}\label{lemma-kappa}
For all $l\geq 0$ the sequence $(\kappa_l(w))_{w\geq 0, 2\mid w}$ is non-negative, increasing and there exists an integer $0\leq \kappa_l\leq 6^{-1}(3+l)(4+l)$ such that for all $w\geq 2l+12$, $\kappa_l(w)=\kappa_l$. Moreover, $\kappa_l(w)=\kappa_l=0$ for all $w\in 2\ZZ$ if $l\in\{0,1,2,3,4\}$.
\end{Lemma}

\begin{proof}
We note that if $w>12$ then $d(w)=1+d(w-12)$. Hence, if $w\geq 2l+12$ we can write $w=w'+12$ with $w'-2i\geq 0$ for all $i\leq l$ and:
$$\kappa_l(w)=l+1+d((l+1)w')-(l+1)-\sum_{i=0}^ld(w'-2i)=\kappa_l(w').$$
From this computation we also see that the sequence is increasing. Moreover, since $d(w)\leq\frac{w}{12}+1$ for all $w\geq 0$ by (\ref{d-w}), we have the trivial upper bound $$\kappa_l(w)\leq \frac{(l+1)(2l+12)}{12}+1$$ if $w\geq 0$ which yields the one for $\kappa_l$. Finally, the fact that $\kappa_l(w)=0$ for all $w\geq 0$ and for all $l\in\{0,1,2,3,4\}$ is trivial for $l=0$ and otherwise follows from the following identities, valid for any $w\in 2\ZZ$:
\begin{eqnarray*}
d(2w)&=&d(w)+d(w-2),\\
d(3w)&=&d(w)+d(w-2)+d(w-4),\\
d(4w)&=&d(w)+d(w-2)+d(w-4)+d(w-6),\\
d(5w)&=&d(w)+d(w-2)+d(w-4)+d(w-6)+d(w-8),\\
\end{eqnarray*}
which can be proved with elementary computations using (\ref{d-w}) and are left to the reader.
\end{proof}

\begin{Remark}
{\em The first coefficients of the sequence $\kappa_l$ are
$$0,0,0,0,0,1,1,2,3,4,5,7,8,10,12,14,16,19\ldots$$
Let $a(n)$ be the cardinality of the set $\{(i,j,k)\in\NN^3:i+2j+3k=n\}$ (set to zero if $n\leq 0$). Then, the first coefficients of the sequence $\kappa_l$ agree with the first coefficients of the sequence $a(n-6)$.
Numerical experiments suggest that 
$$\kappa_l(w)=a(l)-a(l-w/2),\quad \forall l\geq 0,\quad w\in2\ZZ.$$}
\end{Remark}
\begin{Theorem}[Multiplicity estimate]\label{thirdtheorem} 
The following inequality holds:
$$\nu_{\max}(l,w)\leq \delta_l(w)-1+\kappa_l.$$
\end{Theorem}
\begin{proof}
For $\gamma=\sqm{a}{b}{c}{d}\in\Gamma:=\operatorname{SL}_2(\ZZ)$ and $z\in\mathcal{H}$
the complex upper-half plane, we write $$J_\gamma(z)=cz+d\quad\text{ and }L_\gamma(z)=\frac{c}{cz+d}.$$
We will use the identity map $\rho_{1}:\operatorname{SL}_2(\ZZ)\rightarrow \operatorname{SL}_2(\ZZ)$ so that 
$\rho_{1}(\gamma)=\gamma$, 
and its {\em symmetric powers} of order $l\geq 1$:
$$\rho_{l}=S^l(\rho_{1}):\operatorname{SL}_2(\ZZ)\rightarrow \operatorname{SL}_{l+1}(\ZZ),$$
realised in the space of polynomial homogeneous of degree $s=l+1$ with coefficients in $\CC$:
$$\rho_{l}\left(\sqm{a}{b}{c}{d}\right)(X^{s-r}Y^r)=(aX+cY)^{s-r}(bX+dY)^r.$$
For example, for $\gamma$ as above, in the basis $(X^2,XY,Y^2)$:
$$\rho_2(\gamma)=\left(\begin{array}{ccc}a^2 & ab & b^2\\
2ac & ad+bc & 2bd\\ c^2 & cd& d^2\end{array}\right).$$

Let us also consider the derivation $D$ of $\CC((q))$ induced by: $$D=(2\pi i)^{-1}\frac{d}{dz}=q\frac{d}{dq}.$$
Then we have the Ramanujan's differential system:
\begin{equation}\label{ramanujan}
D(E_2)=\frac{1}{12}(E_2^2-E_4),\quad D(E_4)=\frac{1}{3}(E_2E_4-E_6),\quad D(E_6)=\frac{1}{2}(E_2E_6-E_4^2).\end{equation}
Let $f\in\widetilde{M}^{\leq l}_w$ be a quasi-modular form of weight $w$ and depth $\leq l$. There exists, uniquely determined, a polynomial $$P_f\in\operatorname{Hol}(\mathcal{H})[X]$$
of degree equal to the depth of $f$ with coefficients which are holomorphic functions $\mathcal{H}\rightarrow\CC$, such that
for all $\gamma\in\Gamma$ and for all $z\in\mathcal{H}$, we have the functional equation:
\begin{equation}\label{quasimod}
f(\gamma(z))=J_\gamma(z)^w P_f(L_\gamma(z)).\end{equation} It is plain that the coefficients of $P_f$ are quasi-modular forms. In particular, we can identify $P_f$ with an element of $\CC[[q]][X]$. Note that additionally, the constant term of $P_f$ is equal to $f$. We write $Q_f(x)=x^lP_f(1/x)\in\CC[[q]][x]$ and we set:
$${F_f}(z)=\begin{pmatrix}Q_f\\ \frac{\partial Q_f}{\partial x}\\ \frac{\partial^2 Q_f}{\partial x^2}\\
\frac{\partial^3 Q_f}{\partial x^3}\\ \vdots \\ \frac{\partial^{l} Q_f}{\partial x^{l}}\end{pmatrix}_{x=z}.$$
This defines a (weak) vectorial modular form
of weight $w-l$ associated to $\rho_{l}$. In other words, the above is a holomorphic function
$\mathcal{H}\rightarrow\CC^l$ (column) which satisfies the following property: for all $\gamma\in\operatorname{SL}_2(\ZZ)$
and for all $z\in\mathcal{H}$,
$$F_f(\gamma(z))=J_\gamma(z)\rho_{l}(\gamma)\cdot F_f(z).$$
The dimension of the target space of $F_f$ depends on $l$ but not on the depth of $f$ if it is $<l$.
The $D$-Wronskian 
$$W_f(z)=\det({F}(z),D({F})(z),\ldots,D^{l-1}({F})(z))$$
is easily seen to be a modular form of weight $w(l+1)$ and furthermore, we have $\nu(W({F_f}))\geq\nu(f)$.
To fix the ideas of the construction, the reader can check the following formula by using (\ref{ramanujan}): $$W(F_{E_2})=-\frac{E_4}{2\pi i}.$$
If $f\in\widetilde{M}^{\leq l}_w$ is non-zero we have that $Q_f$ has degree $l$ in $X$ and 
the functions $$\left(\frac{\partial^j Q_f}{\partial x^j}\right)_{x= z},\quad j=0,\ldots,l$$
are linearly independent over $\CC$ so that $W(F_f)\in M_{(l+1)w}\setminus\{0\}$.
We now look at $W(F_f)$ with $f=f_{l,w}$. By Lemma \ref{lemma-kappa}:
\begin{eqnarray*}
\delta_l(w)-1&\geq&d((l+1)w)-1-\kappa_l\\
&\geq&\nu(W({F_{f_{l,w}}}))-\kappa_l\\
&\geq&\nu_{\max}(l,w)-\kappa_l.
\end{eqnarray*}

If $l\in\{0,\ldots,4\}$, then $\kappa_l=0$. 
Therefore,
$\nu_{\max}(l,w)=\delta_l(w)-1$.
\end{proof}

\begin{proof}[Proof of Theorem \ref{firsttheorem}]
We suppose that $l\in\{0,1,2,3,4\}$. Lemma \ref{lemma-kappa} tells us that $\kappa_l(w)=\kappa_l=0$
for all $w\in 2\ZZ$, $w\geq 0$. Combining (\ref{l-b}), Theorem \ref{thirdtheorem} we see that $\nu_{\max}(l,w)=\delta_l(w)-1$. Hence an algebraically extremal quasi-modular form of weight $w$ and depth $l$ is also 
analytically extremal.

Let us consider $w\in 2\NN$.
We have the flag of vector spaces
$$M_w\subsetneq\widetilde{M}_w^{\leq 1}\subsetneq\cdots\subsetneq\widetilde{M}_w^{\leq \frac{w}{2}-2}\subsetneq\widetilde{M}_w^{\leq \frac{w}{2}}$$
which implies $\nu_{\max}(0,w)<\nu_{\max}(1,w)<\cdots<\nu_{\max}(\frac{w}{2}-2,w)<\nu_{\max}(\frac{w}{2},w)$.
This means that if $l\neq\frac{w}{2}-1$, an analytically extremal quasi-modular form is algebraically extremal.
\end{proof}

\begin{Remark}{\em For $l\leq 4$ we have $\dim_\CC(\widetilde{M}_w^{\leq l})=\dim_\CC(M_{(l+1)w})$ but we have not constructed natural isomorphisms $\widetilde{M}_w^{\leq l}\cong M_{(l+1)w}$. Hence, 'diagonal bases'
exist for $l\leq 4$ but we have not provided a way to construct them explicitly.
If $V:=\widetilde{M}_w^{\leq l}\neq\{0\}$ the map $f\mapsto W(F_f)$
is homogeneous of weight $l+1$ over $V$ (for example, if $l=1$, it is quadratic). We consider its polarisation
$$V^{\oplus(l+1)}\xrightarrow{\Phi}M_{(l+1)w}.$$
We address the following problem. Characterise the $l$-tuples $(f_1,\ldots,f_l)\in V^{\oplus l}$ such that the map $V\ni f\mapsto \Phi(f,f_1,\ldots,f_l)\in M_{(l+1)w}$ is an isomorphism of $\CC$-vector spaces.}\end{Remark}

\section{Further remarks in depth one}

It belongs to Kaneko and Koike the discovery (in \cite{KK2}) that (algebraically) extremal quasi-modular forms of weight one and two are solutions of linear differential equations belonging to one-parameter families of {\em hypergeometric type}. One of the reasons for which the terminology 'hypergeometric' is used is that moreover, 
the (algebraically) extremal forms (in depth one and two) can also be described inductively by using certain 
{\em contiguity equations} similar to those of Gauss' hypergeometric function first observed by Kaneko and Koike (see Lemma \ref{propo}). The author noticed, in other works, that these contiguity equations can also be viewed as an avatar of an analytic family of Drinfeld modules of rank two. 

From now on, we focus on the case of depth one and weight multiple of six. We write:
$$\Delta=\frac{E_4^3-E_6^2}{1728}\in M_{12}$$ and we set $D=D_1$.
The following is a simple consequence of \cite[Theorem 2.1 part (1)]{KK2} observing that $D(\Delta)=E_2\Delta$ (which is clear from (\ref{ramanujan})). The symbol $\Delta^{1/2}$ denotes the unique normalised square root of $\Delta$ in $u\CC[[q]]$, where $u:=q^{1/2}=e^{\pi i z}$. This can also be viewed as a holomorphic, nowhere vanishing function over $\mathcal{H}$.
\begin{Proposition}\label{prop12}
If $w=6k$, $k\in\NN$, then the function $f_{1,6k}\Delta^{-k/2}$
is the unique solution in $u\CC[[q]]$ of the differential equation
\begin{equation}\label{eqdiff}D^2(X)=\frac{k^2}{4} E_4X.\end{equation}
\end{Proposition}
\begin{proof} First, we acknowledge that this is a direct consequence of the results of \cite{KK2}.
By \cite[Theorem 2.1 part (1)]{KK2} we have 
$$\left(D^2-\frac{w}{6}E_2D+\frac{w(w-1)}{12}D(E_2)\right)(f_{1,w})=0,\quad w\in6\NN,$$
and $f_{1,w}$ is the unique such solution in $\QQ[[q]]$.
Now, in the skew $\CC$-algebra $\CC[E_2,E_4,E_6,\Delta^{-\frac{1}{2}}][D]$ with $Dc=cD+D(c)$, $c\in \CC[E_2,E_4,E_6,\Delta^{-\frac{1}{2}}]$, we have
$$D^2-\frac{w}{6}E_2D+\frac{w(w-1)}{12}D(E_2)=\Delta^{\frac{k}{2}}\left(D^2-\frac{k^2}{4} E_4\right)\Delta^{-\frac{k}{2}}$$ and this concludes the proof.

We can also proceed independently from \cite[Theorem 2.1]{KK2} by using Theorem \ref{firsttheorem} and some basic properties of Rankin-Cohen brackets in the following way.
Note that $\delta_1(6k)=\delta_1(6k+4)=k+1$ for $k\in\NN$. Hence, $\nu_{\max}(1,6k)=\nu_{\max}(1,6k+4)=k$, $k\in\NN$.
Now observe that 
$$g_k:=\left(\Delta^{\frac{k}{2}}\left(D^2-\frac{k^2}{4}E_4\right)\Delta^{-\frac{k}{2}}\right)(f_{1,6k})=\theta^{(1)}_{6k-1}(f_{1,6k})\in\widetilde{M}^{\leq 1}_{6k+4}$$ in the notation of \cite[p. 467]{KK2}, thanks to \cite[Proposition 3.3]{KK2} (the proof of which uses the notion of $\mathfrak{sl}_2$-triple). Hence looking at $q$-expansions and using that $f_{1,6k}=q^k+o(q^{k+1})$ we see that $\nu(g_k)>\nu_{\max}(1,6k+4)$
and $g_k=0$ for all $k$.\end{proof}

A normalised extremal quasi-modular form (of weight $w$ and depth $\leq l$) has its $q$-expansion which is defined over $\QQ$; in other words, the coefficients of its $q$-expansion are rational numbers. By the fact that 
the $q$-expansions of $E_2,E_4,E_6$ are rational integers, it is also clear that the denominators 
of these rational numbers are bounded in absolute value depending on $l$ and $w$. It is then natural to ask for various properties of these rational coefficients such as upper bounds for the primes dividing these denominators in the style of Clausen-von Staudt Theorem.

We recall \cite[Conjecture 2 p. 469]{KK2}:
\begin{Conjecture}[Kaneko and Koike]\label{Conjecture-KK}
If $l\in\{1,2,3,4\}$, then $f_{l,w}$ belongs to $\ZZ[\frac{1}{p}:p<w][[q]]$. Furthermore, except for $f_{1,2}=E_2$, the coefficients of the Fourier series of $f_{l,w}$ are all positive. 
\label{conj1}
\end{Conjecture}
This supports 
Kaneko and Koike's prediction that these coefficients could be the values of some ``counting function of
geometric nature".  
In the direction of this conjecture, we have:
\begin{Theorem}\label{secondtheorem}
For all $k\geq 0$
 we have $f_{1,6k}\in \ZZ[\frac{1}{p}:p<6k][[q]]$.
\end{Theorem}
This result has been also independently noticed by Kaneko, along with the analogue statement for the case $l=2$ of the conjecture, as an application of the techniques introduced in \cite{KK1,KK2}. We propose here to revisit these techniques, including the arguments of \cite{KK2}. The novel observation is the use of certain identities of 'Lax type' (see Lemmas \ref{lemma-F} and \ref{lemma-G} below). We hope that these techniques can contribute to fully solve the conjecture.

We suppose that $k$ is now an indeterminate and we set $B:=\QQ(k)[E_2,E_4,E_6,\Delta^{-\frac{1}{2}}]$.
Note that $B$ embeds in the valued field $\QQ(k)((u))$ where $u^2=q$. 
Also, $B$ and $\QQ(k)((u))$ are endowed with the $\QQ((u))$-linear automorphism $\sigma$
defined by $\sigma(k)=k+1$ and also with the $\QQ(k)$-linear derivation induced by the system (\ref{ramanujan}) and by $D(u)=\frac{u}{2}$. These two $\QQ$-vector space endomorphisms satisfy the commutation rules:
$$D\sigma=\sigma D,\quad Dc=cD+D(c),\quad \sigma c=\sigma(c)\sigma,\quad c\in\QQ(k)((u))$$ and we can consider the skew polynomial Ore algebras
$$B[D,\sigma]\subset\QQ(k)((u))[D,\sigma]$$ which can be identified with sub-algebras of the $\QQ$-linear endomorphisms of $B\subset\QQ(k)((u))$. Let $\Psi$ be in $B[D,\sigma]$. Then, the map
$$\partial_\Psi:B[D,\sigma]\rightarrow B[D,\sigma],\quad \partial_\Psi(x)=[x,\Psi]=x\Psi-\Psi x$$
is a derivation. Also, $B[D,\sigma]$ is equipped with the $\QQ[E_2,E_4,E_6,\Delta^{-\frac{1}{2}}]$-linear automorphism 
$$\Psi\mapsto \Psi^{(1)}:=\sigma \Psi \sigma^{-1}.$$ In other words, $\Psi^{(1)}\in B[D,\sigma]$ is the unique element 
such that $\Psi^{(1)}\sigma=\sigma \Psi$ and clearly, we can also define $\Psi^{(i)}$ for all $i\in\ZZ$. In addition, there is an obvious commutation rule between $X\mapsto X^{(1)}$ and $\delta_{\bullet}$. Similar properties as the above hold in the algebra $\QQ(k)((u))[D,\sigma]$.
We set:
$$E:=D^2-\frac{k^2}{4}E_4\in B[D,\sigma].$$
We are interested in the $\QQ(k)$-vector space
$$H:=\{\Psi\in B[D,\sigma]:\partial_E(\Psi)\in B[D,\sigma]E\}$$
which is the intersection between the left ideal generated by $E$ and the image of $\partial_E$ in $B[D,\sigma]$.
Clearly, $\QQ(k)\subset H$. We set:
$$F:=\Delta^{-\frac{1}{2}}(E_2E_4+(5+6k)E_6+12E_2D)\sigma\in B[D,\sigma].$$
\begin{Lemma}\label{lemma-F}
We have $\QQ(k)[F]\subset H$.
\end{Lemma}
\begin{proof} It suffices to show that $F\in H$. This follows from an elementary computation which uses (\ref{ramanujan}).
More explicitly, we have the formula (our first identity of 'Lax type')
\begin{equation}\label{lax1}
\delta_E(F)=4\Delta^{-\frac{1}{2}}(E_2E_4+2E_6)\sigma E.
\end{equation}
\end{proof}
Let $\mu(k)$ be an element of $\QQ(k)$.
We set
$$G_\mu:=\sigma^2-\mu(k)\left(1-\frac{R}{\Delta^{\frac{1}{2}}}\sigma\right)\in B[\sigma].$$
\begin{Lemma}\label{lemma-G}
The following identity holds:
\begin{equation}\label{lax2}
E^{(2)}G_\mu-G_\mu E=\mu(k)E_4\left(\frac{F}{12}-(k+1)\right).
\end{equation}
\end{Lemma}
\begin{proof} This also follows from an elementary computation, independent this time from (\ref{ramanujan}). Note that the identity is independent of the choice of $\mu(k)$.
\end{proof}
The identity (\ref{lax2}) is our second identity of 'Lax type'.
We now proceed to construct formal solutions of (\ref{eqdiff}). Let $Y$ be a formal solution of 
the equation $D(Y)=\frac{k}{2}Y$ (in terms of analytic functions of two variables, we would have $Y=e^{kz\pi i}$).
Since $\sigma$ and $D$ commute in $\QQ(k)((u))[D,\sigma]$, we can give the field
$$L=\QQ(k)((u))(Y)$$ a structure of $\QQ(k)((u))[D,\sigma]$-module by setting $\sigma(Y)=Yu$. 
If $f\in L$, and $\Psi\in\QQ(k)((u))[D,\sigma]$ we denote by $\Psi(f)$ the action of $\Psi$ on $f$ for this module structure (evaluation). It is easy to see that
$Y$ and $Y^{-1}$ are linearly independent over $\QQ(k)((u))$.

Let $(c_n(x))_{n\geq 0}$ be the sequence of $\QQ(x)$ with $c_0=1$, uniquely defined inductively as follows:
\begin{equation}\label{equations-inductive}
c_n(x)=\frac{240x^2}{n(n+2x)}\sum_{i=1}^n\sigma_3(i)c_{n-i}(x).
\end{equation}
This is just the recursive rule induced by (\ref{eqdiff}) and we have 
$$\operatorname{Ker}(E)=\QQ(k)\varphi_1\oplus\QQ(k)\varphi_2\subset L$$
where
$$\varphi_1=Y^{-1}\sum_{n\geq 0}c_n(-k/2)q^n,\quad \varphi_2=Y\sum_{n\geq 0}c_n(k/2)q^n.$$
In particular, we have 
\begin{equation}\label{varphi2}
\varphi_2=Y\left(1+\frac{60k^2}{k+1}q+o(q)\right)\in Y\QQ(k)((q)).
\end{equation}
It is also easy to see that $\varphi_1$ and $\varphi_2$ are linearly independent over $\QQ(k)((u))$.

\begin{Lemma}\label{lemma-mu} We have
$$G_\mu(\varphi)=Yo(1)$$ in $Y\QQ(k)((q))$ if and only if 
\begin{equation}\label{mu}\mu(k)=\frac{(1+k)(2+k)}{12(7+6k)(11+6k)}.\end{equation}
\end{Lemma}
\begin{proof}
This is straightforward using (\ref{varphi2}). We have
$$\sigma^2(\varphi_2)=qY\left(1+\frac{60(k+2)^2}{k+3}q+o(q)\right).$$
Moreover, we have, by the expansion $\Delta^{-\frac{1}{2}}=u^{-1}+12u+o(u^3)$
(which easily follows from, say, the product formula of $\Delta$),
$$\mu(k)\left(1-\frac{E_6}{\Delta^{\frac{1}{2}}}\sigma\right)(\varphi_2)=\mu(k)Y\left(-\frac{12(7+6k)(11+6k)}{(1+k)(2+k)}q+o(q)\right),$$ and the result follows.
\end{proof}
From now on we set $\mu(k)$ as in (\ref{mu}) and $G=G_\mu$.
\begin{Lemma}\label{propo}
We have $G(\varphi_2)=0$.\end{Lemma}

\begin{proof}
We set $V:=\operatorname{Ker}(E)\cap Y\QQ(k)((u))$; this is a subvector space over $\QQ(k)$ of $L$.
Note that $\varphi:=\varphi_2\in V$ and that $\varphi_1\not\in V$. Indeed otherwise we would have $\varphi_1\in Y\QQ(k)((u))$ contradicting the above-mentioned property that $Y,Y^{-1}$ are linearly independent over $ \QQ(k)((u))$. Hence $V$ is one dimensional, generated by $\varphi$. By Lemma \ref{lemma-F}, $F$ determines an endomorphism of $V=\QQ(k)\varphi$ and therefore $\varphi$ is an eigenvector of $F$ with eigenvalue $\lambda\in\QQ(k)$. By using (\ref{equations-inductive}) this can be easily computed:
$$\lambda=k+1.$$
By Lemma \ref{lemma-G}, $E^{(2)}(G(\varphi))=0$ and
we immediately see that $\sigma^{-2}(G(\varphi))\in\operatorname{Ker}(E)$. Hence,
$G(\varphi)\in\QQ(k)\sigma^2(\varphi)$. We write $G(\varphi)=\eta\sigma^2(\varphi)$ with $\eta\in\QQ(k)$.
By Lemma \ref{lemma-mu}, $\eta=0$.
\end{proof}

\begin{proof}[Proof of Theorem \ref{secondtheorem}] 
By using (\ref{equations-inductive}) and induction, we see that for all $i,j\geq 0$, $\sigma^i(c_j(k/2))\in\QQ(k)$
is regular at $k=0$. Hence, by Proposition \ref{prop12}, for all integer $i\geq 0$, writing 
$$\sigma^i(\varphi_2)|_{k=0}=u^i\sum_{j\geq 0}c_{i,j}q^j,$$ we have
$$f_{1,6i}\Delta^{-i/2}=u^i\sum_{i\geq 0}c_{i,j}q^j.$$ To finish our proof, all we need to show is that
if $p$ is a prime number dividing the denominator of $c_{i,j}$, then $p<6i$.
The property is obvious for $i=0$; indeed, one immediately sees that $\varphi_2|_{k=0}=1$.
For $i=1$ we know that $f_{1,6}=(E_2E_4-E_6)/720$ has $q$-expansion defined over $\ZZ[\frac{1}{2},\frac{1}{3},\frac{1}{5}]$ because $720=2^4\cdot 3^2\cdot 5$. Therefore 
$(\sigma(\varphi_2))|_{k=0}=f_{1,6}\Delta^{-\frac{1}{2}}$ has the required property (notice that $\Delta^{\frac{1}{2}}$ is defined over $\ZZ$).

We now suppose, by induction hypothesis, that for all $i'<i$, if $p$ divides the denominator of $c_{i',j}$ for some
$j$, then $p<6i'$. Since the $q$-expansion of $E_6\Delta^{-\frac{1}{2}}$ is defined over $\ZZ$ and does not depend on $k$, we have that
the $q$-expansion $u^i\sum_{j\geq 0}r_jq^j$ of $(\sigma^{i-2}(\varphi_2)-E_6\Delta^{-\frac{1}{2}}\sigma^{i-1}(\varphi_2))|_{k=0}$
is well defined and has rational coefficients whose primes dividing the denominators do not exceed $6i-6$. But now,
$$c_{i,j}=\frac{i(i-1)}{12(6i-1)(6i-5)}r_j,$$
which implies that if $p$ is a prime dividing the denominator of $c_{i,j}$, then $p<6i$. Notice that the argument must be sligthly
modified if $p=2,3$.
The proof of the Theorem now follows easily by the fact that $f_{1,6i}=\Delta^{i/2}(\sigma^i(\varphi_{2}))|_{k=0}$,
because, as previously noticed, the $q$-expansion of $\Delta^{\frac{1}{2}}$ is defined over $\ZZ$.
\end{proof}
\begin{Remark}{\em We have been unable to show that $f_{1,6i}$ has positive Fourier coefficients for $i\geq 0$. However, from (\ref{equations-inductive}), it is easy to deduce that the Fourier coefficients of 
$f_{1,6i}\Delta^{-i/2}$ are non-negative, for all $i\geq0$.}\end{Remark}

\medskip

We conclude the paper with a prediction. 
In \cite{KK2}, several examples are given, providing experimental evidences of the truth 
of Conjecture \ref{conj1}. We conducted similar numerical examples and we noticed that many normalised extremal quasi-modular forms $f_{l,w}$ with $l\leq 4$ {\em seem} to have integral Fourier coefficients. It is also apparent that this phenomenon ceases to hold for $l>4$. 
Let $\mathcal{E}_l$ be the set of integers $w$ such that $f_{l,w}\in\ZZ[[q]]$. We propose the following, based on numerical investigations we did.

\begin{Conjecture}
If $l\in\{1,2,3,4\}$, then $\mathcal{E}_l$ is an infinite set. If $l>4$, then $\mathcal{E}_l$ is a finite set.
\end{Conjecture}
If $l>4$, we did not find any candidate for an element of $\mathcal{E}_l$.

\section{Appendix by G. Nebe: an example}

In 'small' weight, it is easy to show that an extremal quasi-modular form of depth $1$ has positive integral coefficients. For example, we have used that $f_{1,6}=\frac{D(E_4)}{240}\in\ZZ[\frac{1}{2},\frac{1}{3},\frac{1}{5}][[q]]$ (and the coefficients are positive) but even better this series is in $\ZZ[[q]]$ by (\ref{E2E4E6}). Apart from this and other simple examples, it is not easy to construct extremal quasi-modular forms in $\ZZ[[q]]$ with non-negative coefficients.
As Kaneko and Koike pointed out in \cite{KK2}, the normalised extremal quasi-modular form of weight $6$ and depth $3$
$$f_{3,6}=\sum_{i\geq 2}c_iq^i$$ (with $c_2=1$)
 has the following property. The coefficient $c_d$ is equal to the number of simply ramified coverings of genus $2$ and degree $d$ of an elliptic curve over $\CC$ hence providing another example of normalised extremal quasi-modular form with positive integral coefficients (see \cite{DZ,KZ}).
 
 In this appendix we show that these properties are also shared with the normalised extremal form $f_{1,14}$. Let us consider the {\em theta series} (of weight $12$) associated to the {\em Leech lattice} $\Lambda_{24}$:
$$\theta_{\Lambda_{24}}=E_{12}-\frac{65520}{691}\Delta$$
where $E_{12}$ denotes the normalised Eisenstein series of weight $12$. We have $\theta_{\Lambda_{24}}\in\ZZ[[q]]$ and the coefficients are non-negative. By using the well known fact that $\theta_{\Lambda_{24}}$ is an {\em extremal modular form} of weight $12$ and that $\nu_{\max}(1,14)=2$ (by Theorem \ref{firsttheorem} or by numerical computations),
we easily deduce that:
$$f_{1,14}=\frac{1}{393120}D(\theta_{\Lambda_{24}}).$$ Here, $$393120=2^5\cdot3^3\cdot 5\cdot 7 \cdot13$$ is twice the number of vectors of square norm $2$ in $\Lambda_{24}$
and we see that $f_{1,14}\in\ZZ[\frac{1}{2},\frac{1}{3},\frac{1}{5},\frac{1}{7},\frac{1}{13}][[q]]$ while the coefficients of the Fourier series are all non-negative, in agreement with Conjecture \ref{Conjecture-KK}.
But more is true.

\begin{Theorem}\label{meq1}
We have $f_{1,14}\in\ZZ[[q]]$.
\end{Theorem}

\begin{proof}
	To prove the theorem let $L:=\Lambda _{24}$ be the Leech lattice 
	and $A:=393120=2^5\cdot3^3\cdot 5\cdot 7 \cdot13 $.
	By the above 
	$$f_{1,14} = A^{-1} q \sum _{a=1}^{\infty } a|L_a| q^a  
	\mbox{ where }  L_a = \{ \lambda  \in L \mid \|\lambda  \|^2 = a \} .  $$
	To show the theorem we need to show that $A$ divides $a|L_a|$ for 
	all $a$. We do this prime by prime.  

	Let $G$ be the automorphism group of $L$. 
	Then $G\cong 2.Co_1$ 
	has order $|G|= 2^{22} 3^9 5^4 7^2 11 \cdot 13 \cdot 23 $ 
	and acts on the finite set $L_a$ for all $a>0$.
	For a subgroup $S\leq G$ and $\lambda  \in L_a$ we put 
	$$S\cdot \lambda  := \{ \sigma (\lambda  ) \mid \sigma \in S \} $$ 
	to denote the orbit of $\lambda  $ under $S$. 
	Then $L_a$ is a disjoint union of $S$-orbits and, 
	by Lagrange theorem, 
	$|S \cdot \lambda  | = \frac{|S|}{|U|} $ where 
	$U =  \{ \sigma \in S \mid \sigma (\lambda  ) = \lambda  \}$ is
	the stabiliser in $S$ of $\lambda  $. 

	To see that $5\cdot 7 \cdot 13$ divides $|L_a|$ let 
	$p\in \{5,7,13 \}$. Then the Sylow-$p$-subgroup $S$ of $G$ 
	acts fixed point freely on the non-zero vectors in $L$. 
	So the stabiliser $U$ of any $\lambda  \in L_a$ has index $> 1$ and 
	all $S$-orbits in $L_a$ have length a multiple of $p$.  

	Now let $S$ be a Sylow-$3$-subgroup of $G$, so $S\leq G$ has
	order $3^9$. 
	With Magma \cite{Magma}  we computed the low-index
	 subgroups of $S$ of index $1$, $3$, and $9$.
      For all 101 of these 102 subgroups  $U$, the sublattice of $L$
        consisting of all elements of $L$ that are
        fixed by all  generators of $U$ is $\{ 0 \}$.
	 For the other  subgroup
      $U$ (of index $9$ in $S$) this sublattice is isometric to $\ ^{(3)} A_2$.
      In particular all elements in the lattice have norm divisible by $3$.
      So whenever there is $\lambda  \in L_a$ such that
     $|S\cdot \lambda  | \leq 9$, then $|S\cdot \lambda  | = 9$ and $\lambda  $ lies
     in a sublattice of $L$ that is isometric to $\ ^{(3)} A_2$.
     In particular $a$ is a multiple of $3$. 
     This shows that $9 $ divides $|L_a|$ and $3^3 $ divides $a|L_a|$ for
     all $a>0$.

     To see the divisibility by $2^5$, we could argue the same way, but
     Magma failed to compute the low index subgroups of the Sylow-2-subgroup
     (of order $2^{22}$) of $G$ up to index 16. 
     Instead let $\sigma \in  G$ be an element of order 16 (it is not unique, but 
     the following holds for all three conjugacy classes). 
     Let 
     $$
	     K:=\{ \lambda  \in L \mid \sigma^8(\lambda  ) = \lambda  \} \mbox{ and } 
             N:= \{ \lambda  \in L \mid \sigma ^8(\lambda  )  = - \lambda  \}.
 $$
We compute that 
 $K= \ ^{(2)} E_8$ is the 
 rescaled Barnes-Wall lattice of dimension 8 
  and $N \cong \Lambda _{16}$ is the Barnes-Wall lattice of dimension 16. 
  By \cite[Theorem 14]{roots} all non-zero shells $K_a$ ($a>0$) have
  a cardinality divisible by $2\cdot 8 = 2^4$ and  
  similarly $2\cdot 16 = 2^5 $ divides $|N_a| $ for all $a>0$.
Note also that $K_a = \emptyset $, if $a$ is odd. 

Let $M_a:= L_a \setminus (K_a\cup N_a) $. 
Then 
$$M_a = \{ \lambda  \in L_a \mid |\langle \sigma \rangle \cdot \lambda  | = 16 
\mbox{ and } -\lambda  \not\in \langle \sigma \rangle \cdot \lambda  \} $$ 
is a disjoint union of an even number of $\langle \sigma \rangle $-orbits
of length 16.
In particular $M_a$ has a cardinality divisible by $2\cdot 16 = 2^5$. 
By the above $|N_a|$ is a multiple of $2^5$ and $K_a$ is either 
empty or $a$ is even and $|K_a|$ is a multiple of $2^4$. 
In total we have $a|L_a| = a|M_a| + a|N_a| + a|K_a| $ 
is a multiple of $2^5$.
\end{proof}

Note that the divisibility by $2^53^3$ of the coefficients of 
$D(\Theta _{\Lambda _{24}}) $ also follows from \cite[Theorem 1.2 and 1.3]{K}.


\begin{thebibliography}{99}

	\bibitem{Magma} 
		W. Bosma, J. Cannon, and C. Playoust, 
		{\em The {M}agma algebra system. {I}. {T}he user language}
		J. Symbolic Comput. {\bf 24}, pp. 235-265, (1997).
		
		\bibitem{DZ} R. Dijkgraaf. {\em Mirror symmetry and elliptic curves.} The Moduli Space of Curves, Progress in Math. 129, Birkh\"auser, (1995), 149--163.
		
\bibitem{roots}
	N. Heninger, E. M. Rains, and N. J. A. Sloane,
	{\em On the integrality of {$n$}th roots of generating functions.}
	J. Combin. Theory Ser. A, {\bf 113}, pp. 1732-1745, (2006).

\bibitem{KK1} M. Kaneko \& M. Koike. {\em On Modular Forms Arising from a Differential 
Equation of Hypergeometric Type.} Ramanujan Journ. 7, pp. 145-164, (2003).

\bibitem{KK2} M. Kaneko \& M. Koike. {\em On extremal quasimodular forms.} Kyushu Journal of Math.
Vol. 60, pp. 457-470, (2006).

\bibitem{KZ} M. Kaneko \& D. Zagier. {\em A generalized Jacobi theta function and quasimodular forms.} The Moduli Space of Curves, Progress in Math. 129, Birkh\"auser, (1995), 165--172.

\bibitem{K}  M. Koike. {\em Congruences between extremal modular forms and theta series of special types modulo powers of 2 and 3.} Kyushu Journal of Math.
Vol. 63, pp. 123-132, (2009).

\end{thebibliography}
\end{document}